\documentclass[reqno,12pt]{amsart}

\usepackage{enumerate}
\usepackage[margin=1in]{geometry}
\usepackage{ifpdf}
\usepackage{amsmath}
\usepackage{amsfonts}
\usepackage{amssymb}
\usepackage{amsthm}
\usepackage[ocgcolorlinks,hyperfootnotes=false,colorlinks=true,citecolor=blue,linkcolor=blue,urlcolor=blue]{hyperref}
\usepackage{setspace}
\usepackage{amsrefs}
\usepackage{nicefrac}
\usepackage{graphicx}
\usepackage{color}
\usepackage{mathtools}

\newcommand{\ignore}[1]{}

\renewcommand{\Re}{\operatorname{Re}}
\renewcommand{\Im}{\operatorname{Im}}

\newcommand{\abs}[1]{\left\lvert {#1} \right\rvert}
\newcommand{\sabs}[1]{\lvert {#1} \rvert}
\newcommand{\norm}[1]{\left\lVert {#1} \right\rVert}

\newcommand{\C}{{\mathbb{C}}}
\newcommand{\R}{{\mathbb{R}}}

\newcommand{\sO}{{\mathcal{O}}}

\newcommand{\rank}{\operatorname{rank}}

\newtheorem{thm}{Theorem}[section]

\newtheorem{prop}[thm]{Proposition}

\newtheorem{cor}[thm]{Corollary}

\newtheorem{lemma}[thm]{Lemma}

\theoremstyle{definition}

\theoremstyle{remark}

\author{Ji\v{r}\'{\i} Lebl}
\thanks{The first author was in part supported by NSF grant DMS-1362337 and
Oklahoma State University's DIG and ASR grants.}
\address{Department of Mathematics, Oklahoma State University,
Stillwater, OK 74078, USA}
\email{lebl@math.okstate.edu}

\author{Alan Noell}
\address{Department of Mathematics, Oklahoma State University,
Stillwater, OK 74078, USA}
\email{noell@math.okstate.edu}

\author{Sivaguru Ravisankar}
\address{School of Mathematics, Tata Institute of Fundamental Research,
Mumbai 400005, India}
\email{sivaguru@math.tifr.res.in}

\date{February 23, 2017}

\ifpdf
\hypersetup{
  pdftitle={Codimension two CR singular submanifolds and extensions of CR functions},
  pdfauthor={Jiri Lebl, Alan Noell, and Sivaguru Ravisankar},
  pdfsubject={Several Complex Variables},
  pdfkeywords={Extension of CR functions, Hartogs-Bochner, CR singularity, Levi-flat Plateau problem},
}
\fi

\title{Codimension two CR singular submanifolds and extensions of CR functions}

\keywords{Extension of CR functions, Hartogs-Bochner, CR singularity, Levi-flat Plateau problem}
\subjclass[2010]{32V40 (Primary), 32V25 (Secondary)}

\begin{document}


\begin{abstract}
Let $M \subset \C^{n+1}$, $n \geq 2$, be a real codimension two CR singular
real-analytic submanifold that is nondegenerate and holomorphically flat.
We prove that every
real-analytic function on $M$ that is CR outside the CR singularities
extends to a holomorphic function in a neighborhood of $M$.
Our motivation is to prove the following analogue of the Hartogs-Bochner theorem.
Let $\Omega \subset \C^n \times \R$, $n \geq 2$, be a bounded domain with
a connected real-analytic boundary such that $\partial \Omega$ has only nondegenerate CR singularities.
We prove that if $f \colon \partial \Omega \to \C$ is a real-analytic function that is CR at CR points of $\partial \Omega$, then 
$f$
extends to a holomorphic function on a neighborhood of $\overline{\Omega}$
in $\C^n \times \C$.
\end{abstract}

\maketitle



\section{Introduction} \label{section:intro}

Let $\Omega \subset \C^n \times \R$, $n\ge 2$, be a bounded domain with real-analytic
boundary.  We ask when does a real-analytic function $f \colon \partial
\Omega \to \C$ extend to a function on $\overline{\Omega}$,
holomorphic along the complex directions in $\Omega$?  That is, the extended
function should be a CR function on $\Omega$.

An answer is a so-called Hartogs-Severi theorem,
which is the generalization of the Hartogs extension theorem to $\C^n \times
\R$.  In 1936, Brown~\cite{Brown:36} proved the following statement:
\emph{If $\Omega \subset \C^{n} \times \R$ is a bounded domain
with connected boundary and $f$ is a real-analytic CR function defined on a
neighborhood of $\partial \Omega$, then $f$ extends uniquely to
a real-analytic CR function on a neighborhood of $\overline{\Omega}$.}
Severi~\cite{Severi:32} proved this theorem earlier for $n=1$ under an
additional topological assumption.
Bochner~\cite{Bochner:54} proved a version of this theorem
for harmonic functions in $\C^n \times \R^\ell$.
Further work has been done in proving analogues of such theorems
in more general (that is, not just flat) CR manifolds, see
for example~\cite{HenkinMichel:02}*{Th{\'e}or{\`e}me A}.

If the function $f \colon \partial \Omega \to \C$ extends as a real-analytic
CR function to a whole neighborhood of the boundary, we apply
Hartogs-Severi.  It is classical (also due to Severi)
that at points where $\partial \Omega$ is a CR submanifold,
we get a local holomorphic extension if and only $f$ is CR.
Therefore, it is a necessary
condition for $f$ to be CR on $(\partial \Omega)_{CR}$, the CR points of
$\partial \Omega$.

The question therefore remains: \emph{What happens at the CR
singularities?}
We will study this local extension question for codimension two CR singular
$M \subset \C^{n+1}$.
In this paper
we prove that the condition of holomorphically flat,
that is $M \subset \C^n \times \R$, is sufficient for
the extension as long as the CR singularities
are nondegenerate.  For either a degenerate CR singularity or
a non-flat CR singularity counterexamples exist, see below.
The authors have studied the
smooth case of this problem in \cite{LNR} for elliptic CR singularities.

Let us discuss the local setup and some of the history of the study of
CR singular submanifolds.
Let $M \subset \C^{n+1}$ be a
real codimension two real-analytic submanifold.
CR singular submanifolds of codimension two were first studied in
$\C^2$ by
E.~Bishop~\cite{Bishop65}, who found that
such nondegenerate submanifolds $M$ are locally of the form
\begin{equation}
w = z\bar{z} + \lambda (z^2+\bar{z}^2) + O(3) ,
\end{equation}
where $0 \leq \lambda \leq \infty$ ($\lambda = \infty$ is interpreted
appropriately) is the
so-called \emph{Bishop invariant}.
The work of Bishop in $\C^2$,
especially in the elliptic case ($\lambda < \frac{1}{2}$), has been refined by
Kenig-Webster~\cite{KenigWebster:82},
Moser-Webster~\cite{MoserWebster83},
Moser~\cite{Moser85},
Huang-Krantz~\cite{HuangKrantz95},
and many others, see for example
Huang-Yin~\cite{HuangYin09} and the references therein for recent work.
For work in higher dimensions, especially in codimension two,
see Huang-Yin~\cites{HuangYin09:codim2,HuangYin:flattening1,HuangYin:flattening2},
Gong-Lebl~\cite{GongLebl},
Burcea~\cites{Burcea,Burcea2},
Dolbeault-Tomassini-Zaitsev~\cites{DTZ,DTZ2},
Coffman~\cite{Coffman},
Slapar~\cite{Slapar:16},
and the authors themselves~\cite{LNR}.

A real-analytic codimension two submanifold $M \subset \C^{n+1}$ with a CR singularity at 0
can always be written in suitable
holomorphic coordinates $(z,w) \in \C^n \times \C$ as
\begin{equation}
w = \rho(z,\bar{z}) .
\end{equation}
The submanifold $M$ is said to be \emph{holomorphically flat} if it is a subset
of a real-analytic Levi-flat hypersurface.  In this case, we arrange $\rho$ to
be real-valued.
The submanifold $M$ cannot in general be flattened.
In dimension 3 and higher, existence of such a Levi-flat hypersurface requires
nongeneric conditions on $M$, see \cites{DTZ,HuangYin:flattening1,HuangYin:flattening2}.
A holomorphically flat submanifold $M$ is precisely the condition
$\partial \Omega \subset \C^{n} \times \R$ for the global problem.

Harris~\cite{Harris} studied the
extension of real-analytic CR functions near CR singularities.
For $M \subset \C^m$, Harris provides a
criterion for extension if
$\dim_{\R} M \geq m$, but the condition can be difficult to verify.
Similarly in \cite{LMSSZ} the nonextensibility was studied for
a certain class of CR singular submanifolds.
In particular, a local extension does not always exist even for
nondegenerate manifolds.  The theorem below shows, however, that for
nondegenerate manifolds being
holomorphically flat is a sufficient condition.

Write $M$ as
\begin{equation} \label{eq:Mgrafunnor}
w = A(z,\bar{z}) + B(z,z) + \overline{B(z,z)} + E(z,\bar{z}) ,
\end{equation}
where $A$ and $B$ are quadratic forms and $E$ is $O(3)$.
We say $M$ has a \emph{nondegenerate CR singularity} at the origin
if the Hermitian form $A$ is nondegenerate.
For a CR singular submanifold, let us write $M_{CR}$ for the set of CR points
of $M$.

\begin{thm} \label{thm:mainlocal}
Let $M \subset \C^{n+1}$, $n\ge 2$, be a holomorphically flat real
codimension two
real-analytic submanifold with a
nondegenerate
CR singularity
at $0 \in M$.

Suppose $f \in C^{\omega}(M)$
such that $f|_{M_{CR}}$ is a CR function.
Then there exists a neighborhood $U$ of $0 \in \C^{n+1}$  and
$F \in \sO(U)$ such that $F|_{M \cap U} = f$.
\end{thm}

The hypotheses of the theorem are natural.
Real-analyticity is clearly required for the analytic extension of $F$.
Nondegeneracy is also required since, as mentioned in \cite{LNR},
when $M$ is given by $w = \norm{z}^4$, the function $\sqrt{\Re w}$
is real-analytic and CR on $M$ (on $M$ it is equal to $\norm{z}^2$);
however, it cannot extend to a neighborhood of the origin as a holomorphic
function as such an extension would have to necessarily be a branch of
$\sqrt{w}$.  If $n =1$, the CR condition is vacuous, and the function 
$\bar{z}$ provides a counterexample.  

There do exist CR singular $M$ for which CR functions do not
extend.
For example, in \cite{LMSSZ} it was proved that a CR singular submanifold that
is a diffeomorphic image under a CR map of a CR submanifold admits
non-extensible function.
In light of our theorem above, such images are not nondegenerate and holomorphically flat.

With the local result we obtain the following global version.

\begin{thm} \label{thm:mainglobal}
Let $\Omega \subset \C^{n} \times \R$, $n\ge 2$, be a bounded domain
with real-analytic boundary such that $\partial \Omega$
is connected and
has only
nondegenerate
CR singularities.

\nopagebreak
Let $f \in C^{\omega}(\partial\Omega)$ be
such that $f|_{{(\partial \Omega)}_{CR}}$ is a CR function.
Then there exists an open set $U \subset \C^{n} \times \C$
with $\overline{\Omega} \subset U$ and
$F \in \sO(U)$ such
that $F|_{\partial \Omega} = f$.
\end{thm}

The proof of the theorem follows from the local result stated earlier and the Hartogs-Severi theorem.
Theorem~\ref{thm:mainlocal} implies that $f$ extends from $\partial \Omega$
to a neighborhood, and hence Hartogs-Severi theorem applies.

Let us mention the Levi-flat analogue of the Plateau problem
first studied by Dolbeault-Tomassini-Zaitsev~\cites{DTZ,DTZ2}.  That is,
when does a codimension two submanifold bound a Levi-flat hypersurface?
Using our global result we obtain the following singular solution.

\begin{cor}
Suppose $\Omega \subset \C^{n} \times \R$, $n\ge 2$, is a bounded domain with
real-analytic boundary, and $M = f(\partial \Omega) \subset \C^{n+1}$ is the image
of a real-analytic map $f$ that is CR on ${(\partial \Omega)}_{CR}$.
Suppose $\partial \Omega$ is connected and all
CR singularities of $\partial \Omega$ are nondegenerate.

Then there exists a neighborhood $U \subset \C^{n+1}$ of $\overline{\Omega}$
and a holomorphic map $F \colon U \to \C^{n+1}$ such that
$F|_{\partial \Omega} = f$.  Wherever $F(\overline{\Omega})$ is a smooth real-hypersurface,
it is Levi-flat.
\end{cor}

A natural question is to ask what happens to the theorems in the smooth
case.
While we cannot hope for an extension to a holomorphic function on
$\C^n \times \C$, we may at least hope for a smooth CR extension to
$\C^n \times \R$ along the lines of \cite{LNR}.
In the local case we may hope for an extension to
at least one side as we obtained for elliptic submanifolds.
Furthermore, the Hartogs-Severi theorem fails for smooth CR functions.
Therefore, even if one could extend locally,
for the global theorem to hold in the smooth
case we would require at least some extra topological or geometric
restrictions on
$\partial \Omega$.

The paper has the following structure.  We first study
the normal form for the quadratic part using the results of
Coffman in section \ref{section:quadnormform}.  We also state the
polynomial problem we wish to solve on the quadric model manifolds.
In sections \ref{section:ellipticdir} and
\ref{section:diagonalizable}, we solve the polynomial extension problem
on the model.  In section \ref{section:local} we prove the local extension
result that is the main theorem of this paper.  Finally in section
\ref{section:hartogsseveri} we discuss the failure of the global
extension in the smooth case, and for completeness provide a sketch of a proof the
Hartogs-Severi result in our setting.


\section{Quadratic normal form and the CR singularity} \label{section:quadnormform}

Let $M$ be a
nondegenerate
holomorphically flat
CR singular submanifold written as in \eqref{eq:Mgrafunnor}.  We will call
the submanifold $M^{quad}$ given by
\begin{equation}\label{eq:ModelM}
w = A(z,\bar{z}) + B(z,z) + \overline{B(z,z)} ,
\end{equation}
 the \emph{quadric model} of $M$.

As $A$ is a Hermitian nondegenerate form we diagonalize it and
write $M$ as
\begin{equation}
w = \sum_{j=1}^{\ell} \abs{z_j}^2 -
\sum_{j=\ell+1}^{n} \abs{z_j}^2 +
\sum_{j,k=1}^{n} b_{jk} z_j z_k
+
\sum_{j,k=1}^{n} \bar{b}_{jk} \bar{z}_j \bar{z}_k
+
E(z,\bar{z}) ,
\end{equation}
where we make $b_{jk} = b_{kj}$.  We can arrange for
$\ell \geq \frac{n}{2}$ to make $\ell$ an invariant.

The transformation that diagonalizes $A$ (using $*$-congruence) acts via
congruence on
the matrix $B = [ b_{jk} ]$, which cannot be in general diagonalized.
A normal form for $B$ under linear transformations keeping $A$ diagonal
can be found by classical linear algebra.  The normal form for the two
matrices under biholomorphic transformations requires further computation.
For $n=1$, the normal form for the quadratic part
is the classical Bishop normal form
\begin{equation}
w = \abs{z}^2 +
\lambda (z^2 + \bar{z}^2 ) + E(z,\bar{z}),  \qquad \lambda \geq 0 .
\end{equation}
The number $\lambda$ is a holomorphic invariant, called the
\emph{Bishop invariant}.  For $\lambda < \frac{1}{2}$, $M$ is
called \emph{elliptic},
for
$\lambda = \frac{1}{2}$ it is called
\emph{parabolic}, and for $\lambda > \frac{1}{2}$ it is called
\emph{hyperbolic}.  Normally $\lambda=\infty$ is allowed, although we do
not consider it nondegenerate in our setting.  It would be interpreted
appropriately as $w = z^2 + \bar{z}^2 + E(z,\bar{z})$ and also
called \emph{hyperbolic}.

For $n=2$, that is in $\C^3$, Coffman~\cite{Coffman} computed the quadratic
normal form of all $M$ (including degenerate $M$, although we will not need
these).

We will see that understanding $\C^2$ and $\C^3$ cases are key in understanding
the extension problem.  In $\C^3$, the normal form as per Coffman is one of the
following forms.  First, if $A$ is positive definite, then $B$ can be
diagonalized:
\begin{equation*} \label{eq:typeP}
\tag{P}
w = \abs{z_1}^2 + \abs{z_2}^2
+ \lambda_1 ( z_1^2 + \bar{z}_1^2 )
+ \lambda_2 ( z_2^2 + \bar{z}_2^2 )
+ E(z,\bar{z}), \qquad 0 \leq \lambda_1 \leq \lambda_2 .
\end{equation*}
When $A$ is of mixed signature, Coffman obtains 3 different cases.
First the diagonalizable:
\begin{equation*}\label{eq:typeM.I}
\tag{M.I}
w = \abs{z_1}^2 - \abs{z_2}^2
+ \lambda_1 ( z_1^2 + \bar{z}_1^2 )
+ \lambda_2 ( z_2^2 + \bar{z}_2^2 )
+ E(z,\bar{z}), \qquad 0 \leq \lambda_1 \leq \lambda_2 .
\end{equation*}
Then
\begin{equation*}\label{eq:typeM.II}
\tag{M.II}
w = \abs{z_1}^2 - \abs{z_2}^2
+ \lambda ( z_1z_2 + \bar{z}_1 \bar{z}_2 )
+ E(z,\bar{z}), \qquad \lambda > 0 .
\end{equation*}
And finally
\begin{equation*}\label{eq:typeM.III}
\tag{M.III}
w = \abs{z_1}^2 - \abs{z_2}^2
+ \frac{1}{2} ( z_1^2 + z_2^2 + \bar{z}_1^2 + \bar{z}_2^2 )
+ z_1z_2
+ \bar{z}_1 \bar{z}_2
+ E(z,\bar{z}) .
\end{equation*}

Let us find the CR singular set of the quadric models, that is, when $E=0$.
Write $Q(z,\bar{z})$ for the quadratic part, and therefore the model
$M^{quad}$ is
the submanifold given by $w = Q(z,\bar{z})$.  The CR singularity of
$M^{quad}$ occurs when
$\bar{\partial} (Q-w) = 0$ and
$\bar{\partial} (\bar{Q}-\bar{w}) = 0$.  Therefore it occurs when
$M$ is tangent to a plane where $w$ is constant.  Because $Q$ is quadratic
this only occurs at $w=0$.  As $Q$ is real, the CR singularity occurs
precisely for $w=0$ and the $z$ for which $\bar{\partial} Q = 0$.

Let us start with \eqref{eq:typeP}.  We compute
\begin{equation}
0 = \bar{\partial} Q =
( z_1 + 2 \lambda_1 \bar{z}_1 ) d\bar{z}_1 +
( z_2 + 2 \lambda_2 \bar{z}_2 ) d\bar{z}_2 .
\end{equation}
The equations $z_1 + 2 \lambda_1 \bar{z}_1 = 0$ and $z_2 + 2 \lambda_2 \bar{z}_2 = 0$ have
a unique solution $z_1 = z_2 = 0$ if and only if
$\lambda_1 \not= \frac{1}{2}$ and
$\lambda_2 \not= \frac{1}{2}$.  If $\lambda_j = \frac{1}{2}$, then a CR
singularity also occurs when $\Re z_j = 0$.
In other words, either the CR singularity is
a real 1-dimensional line corresponding to when only one of the $\lambda_j$ is
$\frac{1}{2}$, or it is a totally real 2-dimensional
submanifold of $\{ w = 0 \}$ when both $\lambda_j$ are $\frac{1}{2}$.

We move onto \eqref{eq:typeM.I}.  We compute
\begin{equation}
0 = \bar{\partial} Q =
( z_1 + 2 \lambda_1 \bar{z}_1 ) d\bar{z}_1 +
( - z_2 + 2 \lambda_2 \bar{z}_2 ) d\bar{z}_2 .
\end{equation}
Again, we obtain
$\lambda_1 \not= \frac{1}{2}$ and
$\lambda_2 \not= \frac{1}{2}$ for an isolated singularity, and the same
conclusion for the parabolic $\lambda_j = \frac{1}{2}$ cases.  The
difference is that if $\lambda_2 = \frac{1}{2}$ then the CR singularity is
at points where $\Im z_2 = 0$.

We move onto \eqref{eq:typeM.II}.  We compute
\begin{equation}
0 = \bar{\partial} Q =
( z_1 + \lambda \bar{z}_2 ) d\bar{z}_1 +
( - z_2 + \lambda \bar{z}_1 ) d\bar{z}_2 .
\end{equation}
The two equations $z_1 + \lambda \bar{z}_2 = 0$ and $-z_2 + \lambda \bar{z}_1 = 0$
always have only the unique solution $z_1 = z_2 = 0$, and so the CR
singularity is always isolated.

Finally we move onto \eqref{eq:typeM.III}.  We compute
\begin{equation}
0 = \bar{\partial} Q =
( z_1 + \bar{z}_1 + \bar{z}_2  ) d\bar{z}_1 +
( - z_2 + \bar{z}_2 + \bar{z}_1  ) d\bar{z}_2 .
\end{equation}
Again, the solution $z_1 = z_2 = 0$ is the unique one.

It will be useful to know what the set of CR singularities looks like in all
dimensions, not just $\C^3$.  Let us state what we can say in $n
\geq 2$.

\begin{lemma}\label{lem:singset}
Suppose $M \subset \C^{n+1}$, $n \geq 2$, given by
\begin{equation}
w = A(z,\bar{z}) + B(z,z) + \overline{B(z,z)},
\end{equation}
is a quadric holomorphically flat submanifold with
$A$ nondegenerate.
The set of CR singularities of $M$
is a totally real linear submanifold of the set $\{ w = 0 \}$
of real dimension at most~$n$.
\end{lemma}

\begin{proof}
First diagonalize $A$ and write the submanifold as $w = Q(z,\bar{z})$.  The set
of
CR singularities in the $z$ space is given by $\bar{\partial} Q = 0$.
Since $A$ is diagonal we find that the set is given by $n$ equations of
the form $z_j = c_j \cdot \bar{z}$, where $c_j \in \C^n$.
Since $Q$ is real-valued this is simply the set of critical points
of $Q$ and as $Q$ is quadratic it is a subset of $\{ w = 0 \}$.
The result follows.
\end{proof}

We will say $M$ is \emph{completely parabolic} if the dimension of the set
of CR singularities of $M^{quad}$ is exactly $n$.  When $n=2$, this corresponds to types
\eqref{eq:typeP} and
\eqref{eq:typeM.I} with $\lambda_1 = \lambda_2 = \frac{1}{2}$.

In the next two sections we will prove the following lemma, which has to
be attacked by different techniques, depending on the type of
$M$.

\begin{lemma}\label{lem:PolyExtn}
Suppose $M \subset \C^{n+1}$, $n \geq 2$, given by
\begin{equation}
w = A(z,\bar{z}) + B(z,z) + \overline{B(z,z)},
\end{equation}
is a quadric holomorphically flat submanifold with
a nondegenerate CR singularity at the origin.
Suppose $f(z,\bar{z})$ is a polynomial such that when considered as a
function on $M$ (parametrized by $z$), $f$ is a CR function on
$M_{CR}$.

Then there exists a holomorphic polynomial $F(z,w)$ such that $f$ and $F$
agree on $M$, that~is,
\begin{equation}
f(z,\bar{z}) = F\bigl(z, A(z,\bar{z}) + B(z,z) + \overline{B(z,z)} \bigr) .
\end{equation}
Furthermore, if $f$ is homogeneous of degree $d$, then $F$ is weighted homogeneous of degree $d$, that is,
\begin{equation}
F(z,w)=\sum\limits_{j+2k=d}\, P_j(z)w^k
\end{equation}
where $P_j$ is a homogeneous polynomial of degree $j$.
\end{lemma}

Once the extension exists, the furthermore part of the lemma follows at
once.


\section{Extending along an elliptic direction in the model case}
\label{section:ellipticdir}

In this section we handle the cases when $M$ is not completely parabolic and
is elliptic in some direction, which is the generic case.
Suppose $M \subset \C^{n+1}$, $n \geq 2$, is given by $w = \rho(z,\bar{z})$ for
a real-valued $\rho$.
Let $c \in \C^n$, and define $M_c \subset \C^2$ to be given in the coordinates
$(w,\xi) \in \C^2$ by
\begin{equation} \label{eq:mfldeq}
w = \rho(c \xi, \overline{ c \xi } ) .
\end{equation}
We say $c$ is an \emph{elliptic direction} of $M$, if $M_c$ is an elliptic
submanifold
according to its Bishop invariant.  An important feature of an elliptic
Bishop surface is that it admits a family of attached analytic discs.

\begin{lemma}\label{lem:EllDir}
Suppose $M \subset \C^{n+1}$, $n\ge 2$, is a holomorphically flat submanifold
with a nondegenerate CR singularity at the origin
given by $w = \rho(z,\bar{z})$. Then
$M$ has an elliptic direction
(that is, there exists a $c \in \C^n$ such that $M_c$ is elliptic)
if and only if
\begin{enumerate}[(i)]
\item $n \geq 3$, or
\item $n = 2$ and $M$ is \emph{not} of type \eqref{eq:typeM.I} with
$\lambda_1 = \lambda_2 \geq \frac{1}{2}$.
\end{enumerate}
\end{lemma}

\begin{proof}
First, write $M$ as
\begin{equation}
w = \sum_{j=1}^{\ell} \abs{z_j}^2 -
\sum_{j=\ell+1}^{n} \abs{z_j}^2
+ B(z,\bar{z})
+
E(z,\bar{z}) .
\end{equation}
By a simple linear change of coordinates we ensure that $\ell \geq
n-\ell$, that is, there are more positive than negative eigenvalues in $A$.

If $n \geq 3$, then $\ell \geq 2$.  Therefore if we set $z_3 = \cdots = z_n
= 0$, we obtain a submanifold $\widetilde{M} \subset \C^3$ of
type \eqref{eq:typeP}.  We will show below that such a submanifold always
has an elliptic direction and therefore $M$ has an elliptic direction.

Therefore, suppose that $n=2$.  Let us first dispose with the easy cases.
If \eqref{eq:typeM.I} and $\lambda_1 < \frac{1}{2}$,
then picking $c = (1,0)$ (that is, setting $z_2 = 0$) is
sufficient.  Similarly (same $c$) for cases
\eqref{eq:typeM.II} and
\eqref{eq:typeM.III}.

Now consider \eqref{eq:typeM.I} with $\lambda_1 \geq \frac{1}{2}$.  For
$c \in \C^2$, without loss of generality, assume
$c=(1,a)$ for some $a \in \C$.
Suppose $M$ is given by
\begin{equation}
w = \abs{z_1}^2 - \abs{z_2}^2
+ \lambda_1 ( z_1^2 + \bar{z}_1^2 )
+ \lambda_2 ( z_2^2 + \bar{z}_2^2 )
+ E(z,\bar{z}) .
\end{equation}
Then setting $z = c \xi$ we find that $M_c$ is given by
\begin{equation}
w = (1-\abs{a}^2)\abs{\xi}^2
+ ( \lambda_1  + \lambda_2 a^2) \xi^2
+ ( \lambda_1  + \lambda_2 \bar{a}^2) \bar{\xi}^2
+ E(c \xi,\overline{c \xi}) .
\end{equation}
Therefore the Bishop invariant is
\begin{equation}
\lambda_1
\abs{
\frac{
1  + \frac{\lambda_2}{\lambda_1} a^2
}{
1  - \abs{a}^2
}
} .
\end{equation}
The only way that this will be less than a half is if the second term is
less than 1.  If $\lambda_2 \not= \lambda_1$ we find an $a$ with $\abs{a} \not=1$
that makes the numerator vanish, and therefore we find an elliptic
direction.  So suppose $\lambda_2 = \lambda_1$.  In that case
we find that
\begin{equation}
\abs{
\frac{
1  + a^2
}{
1  - \abs{a}^2
}
} \geq 1 .
\end{equation}
Therefore the Bishop invariant of $M_c$ is always bigger than or equal
to $\lambda_1 = \lambda_2$.  So $M_c$ is elliptic if and only if $\lambda_1 <
\frac{1}{2}$ in this case.

What is left to show is the case \eqref{eq:typeP}.  Take $c \in \C^2$
a unit vector and suppose $M$ is given by
\begin{equation}
w = \abs{z_1}^2 + \abs{z_2}^2
+ \lambda_1 ( z_1^2 + \bar{z}_1^2 )
+ \lambda_2 ( z_2^2 + \bar{z}_2^2 )
+ E(z,\bar{z}) .
\end{equation}
Then setting $z = c \xi$ we find that $M_c$ is given by
\begin{equation}
w = \abs{\xi}^2
+ ( \lambda_1 c_1^2 + \lambda_2 c_2^2) \xi^2
+ ( \lambda_1 \bar{c}_1^2 + \lambda_2 \bar{c}_2^2) \bar{\xi}^2
+ E(c \xi,\overline{c \xi}) .
\end{equation}
The polynomial $P(c_1,c_2) = \lambda_1 c_1^2 + \lambda_2 c_2^2$ must have
a zero on the unit sphere as it has a zero at the origin, and therefore
there exists a $c$ such that $M_c$ is given by
\begin{equation}
w = \abs{\xi}^2 + E(c \xi,\overline{c \xi}) ,
\end{equation}
and therefore is elliptic.
\end{proof}

Let us now focus on the quadric model.
Suppose $M \subset \C^{n+1}$ is given by
\begin{equation}
w = A(z,\bar{z}) + B(z,z) + \overline{B(z,z)}.
\end{equation}
Suppose that $c \in \C^n$ is an elliptic direction.
Pick a real nonzero $w_0$ such that there exists some point $(c \xi,w_0)$
on $M$.  As $w_0$ is nonzero and $M_c$ is elliptic, this means that the
intersection of $M_c$ with $\{ w = w_0 \}$ is an ellipse.
Therefore the map
\begin{equation}
\xi \mapsto (w_0,c \xi)
\end{equation}
induces an analytic disc $\Delta_{c,w_0} \subset \C^{n+1}$ attached to $M$,
that is the boundary $\partial \Delta_{c,w_0} \subset M$.  Furthermore
since the CR singularities are a subset of $\{ w = 0 \}$ by
Lemma~\ref{lem:singset},
then the boundary of the disc is inside the set of CR points of $M$.

As $M_c$ is elliptic, the following can be done either for all $w_0 > 0$
or for all $w_0 < 0$, depending on the sign of the coefficient of $A(c \xi,
\overline{c \xi}) = A(c,\bar{c})\sabs{\xi}^2$.

\begin{lemma}\label{lem:ReaAnaExtnEllDir}
Suppose $M \subset \C^{n+1}$, $n \geq 2$, given by
\begin{equation}
w = A(z,\bar{z}) + B(z,z) + \overline{B(z,z)},
\end{equation}
is a quadric holomorphically flat submanifold with
a nondegenerate
CR
singularity that is not completely parabolic,
and an elliptic direction $c \in \C^n$, and let $w_0$ be
such that $\Delta_{c,w_0}$ is a closed analytic disc attached to $M$.

Suppose $f \colon M \to \C$ is a real-analytic function that is a CR function
on $M_{CR}$.

Then there exists a neighborhood $U \subset \C^{n+1}$ of $\Delta_{c,w_0}$
and a holomorphic function $F \colon U \to \C$ such that $F|_{U \cap M} =
f|_{U \cap M}$.
\end{lemma}

That is, $f$ extends holomorphically to a neighborhood of $\Delta_{c,w_0}$.

\begin{proof}
First notice that since $f$ is real-analytic, then there exists a
neighborhood $V$ of $M_{CR}$ and a holomorphic function $G
\colon V \to \C$, such that $G|_{M_{CR}} = f$.

The proof will follow by constructing a continuous
family of analytic discs $\Delta_t$, all attached to $M$,
such that $\Delta_0$ is a small disc near the CR points such that $\Delta_0
\subset V$, and $\Delta_1 = \Delta_{c,w_0}$.
By the
Kontinuit\"atssatz (see, e.g., \cite{Shabat:book}*{page 189}),
we extend $G$ along this family to a neighborhood of $\Delta_1$.  Because
all the discs are attached to $M$, the extension always agrees with $f$ on
the boundaries.

We begin by considering the submanifold $M_{c,v} \subset \C^2$, for $v \in
\C^n$, given in the $(\xi,w)$ coordinates by
\begin{equation}
w = A(c\xi+v,\overline{c\xi+v}) + B(c\xi+v,c\xi+v) + \overline{B(c\xi+v,c\xi+v)}.
\end{equation}
Clearly $M_{c,0} = M_c$, which is elliptic.  Suppose without loss of
generality that $w_0 > 0$, and therefore $\Delta_{c,w_1}$ exists for all
$w_1 > 0$, but not for any $w_1 < 0$.

Since $\{ w = w_0 \} \cap M_c$ is an ellipse,
$\{ w = w_0 \} \cap M_{c,v}$ is also an ellipse for small $v$. There is also some $w_1 <
0$ such that $\{ w = w_1 \} \cap M_c$ is empty, and therefore we pick
$v$ small enough so that also
$\{ w = w_1 \} \cap M_{c,v}$ is empty.

The set $\{ w = w_0 \} \cap M_{c,v}$ is still a subset of the CR points of $M$.
As $M$ is not completely parabolic,
the set of CR singularities is a totally real linear submanifold of real dimension
at most $n-1$.  Therefore we also make sure $v$ is picked so that the set (the 2
dimensional complex plane)
$\{ (z,w) : z = c\xi+v \text{ for some $\xi \in \C$} \}$ does not
contain any CR singularities of $M$.  That is, no point of $M_{c,v}$ is a CR
singular point of $M$.

Since the defining equation of $M_{c,v}$ is a quadric this means that there
exists some real $w'$ such that
$\{ w = w_2 \} \cap M_{c,v}$ is empty if $w_2 < w'$, and such that
$\{ w = w_2 \} \cap M_{c,v}$ is an ellipse if $w_2 > w'$.  That is, we are
considering the level sets
of a real quadratic function defined on~$\C$:
\begin{equation}
\xi \mapsto A(c\xi+v,\overline{c\xi+v}) + B(c\xi+v,c\xi+v) +
\overline{B(c\xi+v,c\xi+v)} .
\end{equation}
Hence $\{ w = w' \} \cap M_{c,v}$ is a point.  This point
corresponds to a CR point of $M$.

If we pick $w_3 > w'$ very close to $w'$ we obtain an analytic disc
induced by the ellipse
$\{ w = w_3 \} \cap M_{c,v}$, which is completely in $V$.
We can now construct a continuous family of analytic discs attached to $M$ starting with
$\{ w = w_3 \} \cap M_{c,v}$ and ending with
$\{ w = w_0 \} \cap M_{c,v}$.

If we now move $v$ to $0$ we obtain analytic discs attached to $M$ starting
with $\{ w = w_0 \} \cap M_{c,v}$, and ending with the disc induced by
the ellipse $\{ w = w_0 \} \cap M_{c,0}$, which is just $\Delta_{c,w_0}$.

We finish by applying the Kontinuit\"atssatz as mentioned above.
\end{proof}

We now prove a polynomial version of the extension in the model case,
in the case when there is an elliptic direction.  The
case \eqref{eq:typeM.I} with $\lambda_1 \geq \frac{1}{2}$ and the completely
parabolic case will be covered by the next section.

\begin{lemma}\label{lem:PolyExtnEllDir}
Suppose $M \subset \C^{n+1}$, $n \geq 2$, given by
\begin{equation}
w = A(z,\bar{z}) + B(z,z) + \overline{B(z,z)},
\end{equation}
is a quadric holomorphically flat submanifold with
a nondegenerate
CR
singularity that is not completely parabolic, and an elliptic direction $c \in \C^n$.

Suppose $f(z,\bar{z})$ is a polynomial that, when considered as a
function on $M$ (parametrized by $z$), is a CR function on $M_{CR}$.

Then there exists a holomorphic polynomial $F(z,w)$ such that $f$ and $F$
agree on $M$, that~is,
\begin{equation}
f(z,\bar{z}) = F\bigl(z, A(z,\bar{z}) + B(z,z) + \overline{B(z,z)} \bigr) .
\end{equation}
\end{lemma}

\begin{proof}
Let us consider $M_c$.  Suppose without loss of generality that
$\Delta_{c,w_0}$ exists for all $w_0 > 0$.

If $c \in \C^n$ is an elliptic direction, then we notice that a whole
neighborhood $C \subset \C^n$ of $c$ are also elliptic directions for $M$.

Let $c' \in C$.  In \cite{LNR}*{Lemma 5.1}, we proved that a polynomial $P(\xi,\bar{\xi})$, when considered as
a function on $M_{c'}$ extends holomorphically to a polynomial in the
variables $\xi$ and $w$ as long as it extends along each analytic disc.
Since $f$ extends by the previous lemma to each $\Delta_{c',w_0}$ for $w_0 >
0$ and $c' \in C$ we can use this lemma.  Therefore we obtain a polynomial
$F_{c'}(\xi,w)$ that extends $f(c' \xi, \overline{c' \xi})$.

Now we look at $w_0 = 1$ and $c$.  There exists a holomorphic function $F$
defined in a neighborhood $U \subset \C^{n+1}$ of $\Delta_{c,1}$ that extends $f$.
Since for all $c' \in C$,
\begin{equation}
(\xi,w) \mapsto F(c'\xi, w)
\end{equation}
agrees with $F_{c'}(\xi,w)$ on some open set, and since $C$ is an open set,
we find that $F(z,w)$ agrees with a polynomial on an open set and we are
done.  See \cite{LNR}*{Proposition 5.2}.
\end{proof}


\section{Extending in the diagonalizable model case}
\label{section:diagonalizable}

There are two cases of model manifolds that we have not covered yet.
First, $n=2$ and $M$ is of type
\eqref{eq:typeM.I} with both $\lambda_j$ large,
that is, larger than $\frac{1}{2}$.  The second case is when $M$ is
a completely parabolic submanifold and $n > 2$.
We first focus on the extension in $n=2$, and then use this result to
solve the completely parabolic case when $n > 2$.

The following technique in $n=2$ works equally well for
\eqref{eq:typeP} and \eqref{eq:typeM.I} and both $\lambda_j > 0$, and
therefore we state the results for both types.
Let $M$ be given by \eqref{eq:typeP} or
\eqref{eq:typeM.I} with $E=0$.
That is,
\begin{align}
\label{eq:3dtypeP}
w&=\abs{z_1}^2+\abs{z_2}^2 + \lambda_1\left(z_1^2+\bar{z}_1^2\right) + \lambda_2\left(z_2^2+\bar{z}_2^2\right)
\intertext{or}
\label{eq:3dtypeM.I}
w&=\abs{z_1}^2-\abs{z_2}^2 + \lambda_1\left(z_1^2+\bar{z}_1^2\right) + \lambda_2\left(z_2^2+\bar{z}_2^2\right)
\end{align}
for some $0 < \lambda_1\leq\lambda_2$.  Note that we have already handled
the case when $\lambda_1 = 0$.

\begin{prop} \label{prop:modeldiagext}
Suppose $M\subset\C^2\times\C$, given by \eqref{eq:3dtypeP} or
\eqref{eq:3dtypeM.I}, $0 < \lambda_1 \leq \lambda_2$.

Let $CR^d(M)$ be the (complex) vector space of degree $d$ homogeneous
polynomials
$f(z,\bar{z})$, which when considered as
functions on $M$ (parametrized by $z$), are CR functions on
$M_{CR}$.

Then
\begin{equation}
\dim CR^d(M) \leq
\left\lfloor
\frac{(d+2)^2}{4}
\right\rfloor .
\end{equation}
\end{prop}

We will see below that the inequality is in fact an equality.
Before we prove this, we need a small claim.

\begin{prop}
Let $A_\ell$ be an $\ell \times \ell$ matrix with zeros everywhere
except exactly on the super and sub diagonal.  That is, if $A_\ell = [a_{jk}]$
then $a_{jk} \not= 0$ if and only if $j=k+1$ or $j=k-1$.  Then
\begin{equation}
\rank A_\ell = 2 \left\lfloor \frac{\ell}{2} \right\rfloor
=
\begin{cases}
\ell & \text{if $\ell$ even,} \\
\ell-1 & \text{if $\ell$ odd.}
\end{cases}
\end{equation}
\end{prop}

\begin{proof}
This is easily checked for $\ell = 1,2,3$.  Then after a row and a
column operation we find that $A_\ell$ becomes $\left[\begin{smallmatrix} 0 & 1\\ 1
& 0 \end{smallmatrix} \right] \oplus A'_{\ell-2}$, where $A'_{\ell-2}$ is an
$(\ell-2)$ by $(\ell-2)$ matrix of the same type.  The result follows.
\end{proof}

\begin{proof}[Proof of Proposition \ref{prop:modeldiagext}]
Let $Q(z,\bar{z}) = A(z,\bar{z}) + B(z,z) + \overline{B(z,z)}$.
Let $L$ be the CR vector field on $M$
\begin{equation}
L =
Q_{\bar{z}_2} \frac{\partial}{\partial \bar{z}_1} -
Q_{\bar{z}_1} \frac{\partial}{\partial \bar{z}_2}
=
(\epsilon z_2 + 2 \lambda_2 \bar{z}_2) \frac{\partial}{\partial \bar{z}_1} -
(z_1 + 2 \lambda_1 \bar{z}_1) \frac{\partial}{\partial \bar{z}_2}
\end{equation}
where $\epsilon = \pm 1$.  Let us apply $L$ to a monomial:
\begin{equation} \label{eq:Lapplied}
\begin{split}
L (
z_1^{a_1} z_2^{a_2}
\bar{z}_1^{b_1} \bar{z}_2^{b_2}
)
=
&
\epsilon
b_1
z_1^{a_1} z_2^{a_2+1}
\bar{z}_1^{b_1-1} \bar{z}_2^{b_2}
+
2 \lambda_2
b_1
z_1^{a_1} z_2^{a_2}
\bar{z}_1^{b_1-1} \bar{z}_2^{b_2+1}
\\
&
-
b_2
z_1^{a_1+1} z_2^{a_2}
\bar{z}_1^{b_1} \bar{z}_2^{b_2-1}
-
2 \lambda_1
b_2
z_1^{a_1} z_2^{a_2}
\bar{z}_1^{b_1+1} \bar{z}_2^{b_2-1} .
\end{split}
\end{equation}
A homogeneous polynomial $f$ is CR if and only if $Lf = 0$, which is a linear equation in the
coefficients of $f$.  Ordering the coefficients somehow, let $c$ be the
vector of coefficients of $f$.  The equation $Lf = 0$ can be written as
a matrix equation $Xc = 0$.  From the above equation we see that each column
has at most 4 nonzero entries.  Two of those entries correspond to
monomials with the same degree of $\bar{z}$ and $z$ and the other two
entries correspond to raising the degree of $z$ by one and lowering the degree
of $\bar{z}$ by one.

We order the monomials by an ordering that satisfies
$z_1^{a_1} z_2^{a_2}
\bar{z}_1^{b_1} \bar{z}_2^{b_2} <
z_1^{a_3} z_2^{a_4}
\bar{z}_1^{b_3} \bar{z}_2^{b_4}$
if $a_1+a_2 < a_3+a_4$ or if
$a_1+a_2 = a_3+a_4$ and
$a_1 < a_3$, or if
$a_1+a_2 = a_3+a_4$ and
$a_1 = a_3$ and
$b_1 < b_3$.  For example for $d=2$ we order the monomials as
\begin{equation}
\bar{z}_1^2 <
\bar{z}_1\bar{z}_2 <
\bar{z}_2^2 <
z_1 \bar{z}_1 <
z_1 \bar{z}_2 <
z_2 \bar{z}_1 <
z_2 \bar{z}_2 <
z_1^2 <
z_1z_2 <
z_2^2 .
\end{equation}
We ignore the part of $Lf$ that lowers the degree of $\bar{z}$.  That
operation moves ``to the right'' in the columns of $X$.
If we now decompose the matrix $X$ into blocks, for each fixed
monomial in $z$, we find blocks of the form $A_\ell$ as above,
with zeros to the right.

Again for example for second degree with the monomials ordered as above we
have the following matrix.  The blocks of the form $A_\ell$ are boxed.
The submanifold is of type \eqref{eq:typeM.I}.
\begin{equation}
\def\mcr#1{\multicolumn{1}{c|}{#1}}
\def\mcl#1{\multicolumn{1}{|c}{#1}}
\left[
\,
\begin{array}{cccccccccc}
\cline{1-3}
\mcl{0}          & -2 \lambda_1 & \mcr{0}            & 0           & 0            & 0           & 0            & 0 & 0 & 0 \\
\mcl{4\lambda_2} & 0            & \mcr{-4 \lambda_1} & 0           & 0            & 0           & 0            & 0 & 0 & 0 \\
\mcl{0}          & 2 \lambda_2  & \mcr{0}            & 0           & 0            & 0           & 0            & 0 & 0 & 0 \\
\cline{1-5}
0          & -1           & 0            & \mcl{0}           & \mcr{-2 \lambda_1} & 0           & 0            & 0 & 0 & 0 \\
0          & 0            & -2           & \mcl{2 \lambda_2} & \mcr{0}            & 0           & 0            & 0 & 0 & 0 \\
\cline{4-7}
-2         & 0            & 0            & 0           & 0            & \mcl{0}           & \mcr{-2 \lambda_1} & 0 & 0 & 0 \\
0          & -1           & 0            & 0           & 0            & \mcl{2 \lambda_2} & \mcr{0}            & 0 & 0 & 0 \\
\cline{6-7}
0          & 0            & 0            & 0           & -1           & 0           & 0            & 0 & 0 & 0 \\
0          & 0            & 0            & -1          & 0            & 0           & -1           & 0 & 0 & 0 \\
0          & 0            & 0            & 0           & 0            & -1          & 0            & 0 & 0 & 0 \\
\end{array}
\,
\right]
\end{equation}

For the monomials of a fixed total degree $j$ of $\bar{z}$, for $j \geq 1$,
we will find $(d-j+1)$ blocks of the form $A_{j+1}$ in the matrix.
That is for each fixed monomial $z^\alpha$, for $\abs{\alpha} = d-j < d$,
there is a block $A_{j+1}$ in the columns corresponding to
monomials $z^{\alpha} \bar{z}^{\beta}$ where $\abs{\beta} = j$.
There are always $(d-j+1)$ monomials of the form $z^\alpha$, and that is why
we get that many blocks.

For the example above, $d=2$, there is precisely one $A_3$ type block
for the monomials corresponding to monomials that are quadratic in
$\bar{z}$, that is, monomials of the form $\bar{z}^\beta$ where
$\abs{\beta} = 2$.  There are also two blocks, of type $A_2$.  One for the
monomials of the form $z_1 \bar{z}^\beta$ and one for the monomials
of the form $z_2 \bar{z}^\beta$, where $\abs{\beta} = 1$.

Therefore the rank of the matrix is at least
$\sum_{j=1}^d (d-j+1) 2 \left\lfloor \frac{j+1}{2} \right\rfloor$, and
\begin{equation}
\dim CR^d(M) \leq
\binom{d+3}{3} - \sum_{j=1}^d 2 (d-j+1)
\left\lfloor \frac{j+1}{2} \right\rfloor .
\end{equation}

It is not difficult to show that
\begin{equation}
\binom{d+3}{3} - \sum_{j=1}^d 2 (d-j+1)
\left\lfloor \frac{j+1}{2} \right\rfloor
=
\left\lfloor
\frac{(d+2)^2}{4}
\right\rfloor .
\end{equation}
\end{proof}

\begin{prop} \label{prop:modeldiagextpoly}
Suppose $M\subset\C^2\times\C$, given by \eqref{eq:3dtypeP} or
\eqref{eq:3dtypeM.I}, $0 < \lambda_1 \leq \lambda_2$.
Then the space $CR^d(M)$ is spanned by monomials in $z$ and
$w=A(z,\bar{z}) + B(z,z) + \overline{B(z,z)}$.

Consequently, suppose
$f(z,\bar{z})$ is a polynomial such that when considered as a
function on $M$ (parametrized by $z$), $f$ is a CR function on
$M_{CR}$.
Then there
exists a holomorphic polynomial $F(z,w)$ such that $f$ and $F$
agree on $M$, that~is,
\begin{equation}
f(z,\bar{z}) = F\bigl(z, A(z,\bar{z}) + B(z,z) + \overline{B(z,z)} \bigr) .
\end{equation}
\end{prop}

\begin{proof}
Weighted homogeneous polynomials in $z$ and $w$ when restricted to $M$ give a CR function, that is
elements of $CR^d(M)$.  Therefore the proof consists of showing that
the dimension of this set is the same as the dimension of $CR^d(M)$.

The dimension of the space weighted homogeneous polynomials of degree $d$ is
\begin{equation}
\sum\limits_{j+2k=d} (j+1)
=
\sum_{k=0}^{\lfloor d/2 \rfloor}
(d-2k+1)
=
\left\lfloor
\frac{(d+2)^2}{4}
\right\rfloor .
\end{equation}
Since the space $CR^d(M)$ cannot be any larger we find that the weighted
homogeneous polynomials in $z$ and $w$ will span the space.

If $f(z,\bar{z})$ is any polynomial that is CR on $M_{CR}$, then \eqref{eq:Lapplied}
implies that the homogeneous parts of $f$ are CR, and hence are in the span
of $z$ and $w = A(z,\bar{z}) + B(z,z) + \overline{B(z,z)}$.
\end{proof}

Now that we have covered all the $n=2$ cases, we can prove the
completely parabolic case in $n > 2$.

\begin{prop}
Suppose $M\subset\C^{n+1}$, $n> 2$, given by
$w=A(z,\bar{z}) + B(z,z) + \overline{B(z,z)}$ is completely parabolic (that is, the
real dimension of the set of CR singularities is exactly $n$).
Suppose $f(z,\bar{z})$ is a polynomial such that when considered as a
function on $M$ (parametrized by $z$), $f$ is a CR function on
$M_{CR}$.

Then there
exists a holomorphic polynomial $F(z,w)$ such that $f$ and $F$
agree on $M$, that~is,
\begin{equation}
f(z,\bar{z}) = F\bigl(z, A(z,\bar{z}) + B(z,z) + \overline{B(z,z)} \bigr) .
\end{equation}
\end{prop}

\begin{proof}
Write $Q(z,\bar{z}) = A(z,\bar{z}) + B(z,z) + \overline{B(z,z)}$.
Let us choose a 2 complex dimensional plane through the origin in the $z$
variables.  That is, choose a linear $L \colon \C^2 \to \C^n$ and define
an $M_L \subset \C^2 \times \C$ by
\begin{equation}
w = Q(L\xi,\overline{L\xi})
\end{equation}
for variables $(\xi,w) \in \C^2 \times \C$.  Suppose that $M_L$ is
nondegenerate, which is true for an open dense set of $L$.

As the CR singular set of $M$ is a maximally totally real $n$-dimensional
plane in the $z$-space, $M_L$ is also going to be completely parabolic.

Take $L$ of the form
\begin{equation}
\left[\begin{array}{ccc}
1&0\\
0&1\\
\omega_1 & \omega_2
\end{array}\right]
\end{equation}
for some column vectors $\omega_j\in\C^{n-2}$.
We apply the polynomial result, Proposition~\ref{prop:modeldiagextpoly}.
For
every $L$ as above, we
find a polynomial $F_L(\xi,w)$ such that
\begin{equation} \label{eqn:polyextnslices}
f(L\xi,\overline{L\xi}) = F_L\bigl(\xi,Q(L\xi,\overline{L\xi})\bigr)
= \sum\limits_{2j+\abs{\alpha}=d}\, c_{\alpha
j}(\omega_1,\omega_2,\bar{\omega}_1,\bar{\omega}_2)\xi^\alpha Q^j .
\end{equation}
To see that the $c_{\alpha j}$'s are independent of $\bar{\omega}$ we use
the CR vector field of $w=Q(L\xi,\overline{L\xi})$.
Without loss of generality suppose that $\omega_1=0$.  The argument is the
same for higher dimensions, so for simplicity assume $n=3$, that is $\omega_2
\in \C$.
Then, consider the CR vector field $X$ of
\begin{equation}
w=Q(\xi_1,\xi_2,\omega_2\xi_2,\bar{\xi}_1,\bar{\xi}_2,\overline{\omega_2\xi_2})
=
\widetilde{Q}(\xi_1,\xi_2,\omega_2,\bar{\xi}_1,\bar{\xi}_2,\bar{\omega}_2)
\end{equation}
given by
\begin{equation}
X=\widetilde{Q}_{\bar{\xi}_2}\frac{\partial}{\partial\bar{\omega}_2} -
\widetilde{Q}_{\bar{\omega}_2}\frac{\partial}{\partial\bar{\xi}_2}
=\left(Q_{\bar{z}_2}+Q_{\bar{z}_3}\bar{\omega}_2\right)\frac{\partial}{\partial\bar{\omega}_2}
- Q_{\bar{z}_3}\bar{\xi}_2\frac{\partial}{\partial\bar{\xi}_2}.
\end{equation}
Since $f\vert_{M_L}$ is CR, applying $X$ to \eqref{eqn:polyextnslices} gives
us that the $c_{\alpha j}$'s are independent of $\bar{\omega}$.  That is
because $\xi^\alpha Q^j$ are CR and hence $X$ only hits the $c_{\alpha j}$.

Keeping with $n=3$ for simplicity, we have
\begin{equation}
f(
\xi_1,\xi_2, \omega_2 \xi_2,
\bar{\xi}_1,\bar{\xi}_2, \bar{\omega}_2 \bar{\xi}_2
)
= \sum\limits_{2j+\abs{\alpha}=d}\, c_{\alpha
j}(0,\omega_2)\xi^\alpha Q^j .
\end{equation}
Setting $\xi_1 = z_1$,
$\xi_2 = z_2$,
$\omega_2 = \frac{z_3}{z_2}$, we obtain
\begin{equation}
f(z,\bar{z})
= \sum\limits_{2j+\abs{\alpha}=d}\, c_{\alpha
j}\biggl(0,\frac{z_3}{z_2}\biggr)z_1^{\alpha_1}z_2^{\alpha_2} Q^j .
\end{equation}
We have found a rational extension to $\C^3 \times \C$,
with a possible pole when $z_2 = 0$.  Doing the same argument with $\omega_2
= 0$, we find another rational extension with a possible pole at $z_1 = 0$.
Outside any poles the extensions must be identical as the extension near CR
points is unique as a holomorphic function.  The poles are therefore
only on the set $z_2 = z_1 = 0$, which means there are no poles.  We obtain
a polynomial extension $F(z,w)$.
\end{proof}


\section{Local Extension}
\label{section:local}

In this section, we prove Theorem \ref{thm:mainlocal}, the local version of the main theorem.

\begin{prop}\label{prop:FormalPowerSeries}
Let $M \subset \C^{n+1}$, $n\ge 2$, be a holomorphically flat real
codimension two
real-analytic submanifold with a
nondegenerate
CR singularity
at $0 \in M$.
Suppose $M$ is defined by $w = \rho(z,\bar{z})$, for $(z,w) \in \C^{n} \times \C$,
with $\rho$ real-valued, $\rho(0)=0$, and  $d\rho(0) = 0$.

Suppose $f \in C^{\omega}(M)$
such that $f|_{M_{CR}}$ is a CR function.
There exists a formal power series $F(z,w)$ for $f$ at the origin,
that is, $F\bigl(z,\rho(z,\bar{z})\bigr)$ is equal to $f$ formally at the
origin.
\end{prop}

\begin{proof}
Write $M$ as
\begin{equation}
w = \rho(z,\bar{z}) = Q(z,\bar{z}) + E(z,\bar{z}) ,
\end{equation}
where $E$ is $O(3)$.
Parametrizing $M$ by $z$, decompose $f$ using the $z$ variables
into homogeneous parts
\begin{equation}
f(z,\bar{z})
=
f_k(z,\bar{z}) + \sum_{j=k+1}^\infty f_j(z,\bar{z})
= f_k(z,\bar{z}) + \widetilde{f}(z,\bar{z}) .
\end{equation}

Take a CR vector field $X$ on $M$ and the corresponding CR vector field
$X^{quad}$ on $M^{quad}$:
\begin{align}
X &=
\left(
Q_{\bar{z}_j}
+
E_{\bar{z}_j}
\right)
\frac{\partial}{\partial \bar{z}_k}
-
\left(
Q_{\bar{z}_k}
+
E_{\bar{z}_k}
\right)
\frac{\partial}{\partial \bar{z}_j} ,
\\
X^{quad} &=
Q_{\bar{z}_j} \frac{\partial}{\partial \bar{z}_k}
-
Q_{\bar{z}_k} \frac{\partial}{\partial \bar{z}_j} .
\end{align}
Then as $f$ is CR
\begin{equation}
0 = Xf =
X (f_k + \widetilde{f}) =
\left(
Q_{\bar{z}_j}
\right)
\frac{\partial f_k}{\partial \bar{z}_k}
-
\left(
Q_{\bar{z}_k}
\right)
\frac{\partial f_k}{\partial \bar{z}_j}
+
O(k+1)
=
X^{quad} f_k + O(k+1) .
\end{equation}
Therefore $X^{quad} f_k = 0$ and
$f_k(z,\bar{z})$ is a CR function on the model $M^{quad}$
given by $w = Q(z,\bar{z})$.  By Lemma~\ref{lem:PolyExtn}, we can thus write
$f_k(z,\bar{z}) = F_k\bigl(z,Q(z,\bar{z})\bigr)$ for some
weighted homogeneous $F_k(z,w)$.
Now $F_k\bigl(z,\rho(z,\bar{z})\bigr)$ is a CR function on $M$ whose
terms of degree $k$ are precisely $f_k$.  Therefore
the difference
\begin{equation}
f(z,\bar{z}) - F_k\bigl(z,\rho(z,\bar{z})\bigr)
\end{equation}
is a CR function of one higher order, and so we obtain a
formal power series.
\end{proof}

Next we prove the convergence of the formal power series, first in the case $n=1$.

\begin{lemma}\label{lem:ConfFormalPowerSeries}
Let $M \subset \C^{2}$ be a
real codimension two
real-analytic submanifold with a
nondegenerate
CR singularity
at $0 \in M$ defined by $w = \rho(z,\bar{z})$, for $(z,w) \in \C^2$,
with $\rho(0) = 0$ and $d\rho(0) = 0$.

Suppose $f \in C^{\omega}(M)$ admits a formal power series
$F(z,w)$, that is,
$F\bigl(z,\rho(z,\bar{z})\bigr)$ is equal to $f$ formally.
Then $F$ is convergent.
\end{lemma}

\begin{proof}
Parametrizing $M$ by $z$, we write $f(z,\bar{z})$ for the value of $f$ on $M$ at
$\bigl(z,\rho(z,\bar{z})\bigr)$ as usual.
We may locally complexify and treat $z$ and $\bar{z}$ as independent
variables.

\textbf{Case 1:}
$\rho(z,0) \not\equiv 0$.

After a change of variables, successively taking changes of coordinates sending
$z$ to $z + a z^j w^k$, such an $M$ can be represented
by
\begin{equation}
w=\abs{z}^2+c\left(z^k+\bar{z}^k\right) + O(k+1)
\end{equation}
for some $c>0$ and $k\ge 2$.
Using Weierstrass, we may locally solve for $\bar{z}$ in terms of $w$ and
$z$.  Let us denote these solutions by $\xi_1(z,w)$, \ldots, $\xi_k(z,w)$; one of these is the complex conjugate of $z$.
Let $\xi(z,w)$ be any of the $\xi_j$'s.
So,
\begin{equation} \label{eq:avgout}
\widetilde{F}(z,w) = \frac{1}{k}\sum\limits_{j=1}^k f\bigl(z,\xi_j(z,w)\bigr)
\end{equation}
is a well-defined holomorphic function in a neighborhood of the origin
because the right hand side is a symmetric function of the $\xi$.
See e.g.\ \cite{Whitney:book}*{Lemma 8A in chapter 1}.

When $w=\rho(z,\bar{z})$, we have
\begin{equation}
f(z,\xi)\overset{formally}{=}F(z,\rho(z,\xi))=F\bigl(z,\rho(z,\bar{z})\bigr)\overset{formally}{=}f(z,\bar{z}).
\end{equation}
By formally we mean that we are only testing that the equality holds up to any finite order.
Since $f$ is real-analytic, we obtain $f(z,\xi) = f(z,\bar{z})$ when $w =
\rho(z,\bar{z})$.  That is, $f$ is invariant under replacing $\bar{z}$ with any of the $\xi_j$'s.
In other words, when $w = \rho(z,\bar{z})$, then in \eqref{eq:avgout} all
terms in the sum are equal to $f(z,\bar{z})$ and therefore
\begin{equation}
\widetilde{F}\bigl(z,\rho(z,\bar{z})\bigr) = f(z,\bar{z}).
\end{equation}
The formal power series for $f$ in terms of $z$ and $w$ must be unique, so
$F(z,w) = \widetilde{F}(z,w)$ as power series, and therefore $F$ converges.

\textbf{Case 2.} Suppose $\rho(z,0) \equiv 0$.

The submanifold $M$ has an infinite Moser invariant.
Moser~\cite{Moser85} proved there exists a local
biholomorphic change of variables near zero so that $M$ is given by
\begin{equation}
w = z\bar{z} .
\end{equation}
Write
\begin{equation}
f(z,\bar{z}) =
\sum c_{k,j} z^k {\bar{z}}^j .
\end{equation}
Cauchy estimates give
$\sabs{c_{k,j}} \leq \frac{M}{\epsilon^{k+j}}$ for some $\epsilon > 0$.
As the power series $F$ can be written in terms of $z$ and $w = z\bar{z}$, we note that
$c_{k,j} = 0$ if $k < j$.  That is writing $d_{k,j} = c_{k+j,j}$
\begin{equation}
F(z,w) = \sum_{j,k} d_{k,j} z^kw^j .
\end{equation}
Because
$\sabs{d_{k,j}} = \sabs{c_{k+j,j}} \leq \frac{M}{\epsilon^{k}
\epsilon^{2j}}$, the series converges.
\end{proof}

We can now finish the proof of Theorem~\ref{thm:mainlocal}.

\begin{proof}[Proof of Theorem~\ref{thm:mainlocal}]
By the Proposition~\ref{prop:FormalPowerSeries} we obtain
a formal power series $F(z,w)$, such that
$f(z,\bar{z}) = F\bigl(z,\rho(z,\bar{z})\bigr)$.

Given any nonzero $c \in \C^n$ we note that using coordinates
$(\xi,w) \in \C \times \C$, we have an $M_c$
\begin{equation}
w = \rho(c\xi,\overline{c\xi}) ,
\end{equation}
and an $f(c\xi,\overline{c\xi})$, which has a formal power series $F(c\xi,w)$
at the origin.

The submanifold $M_c$ is nondegenerate for an open dense set of $c \in \C^n$.
Therefore $F(c\xi,w)$ converges by the Lemma~\ref{lem:ConfFormalPowerSeries} for an open dense set
of $c \in \C^n$.
Therefore $F(z,w)$ converges
via a standard Baire category argument
(see e.g.\ \cite{BER:book}*{Theorem 5.5.30}).
\end{proof}


\section{Hartogs-Severi}
\label{section:hartogsseveri}

The Hartogs-Severi result cannot be simply extended to the smooth case.
This also means that Theorem~\ref{thm:mainglobal} fails in the
smooth case.  Counterexamples have appeared in the literature, but
let us discuss a simple counterexample as it pertains to our setup.

Let $U \subset \C^2 \times \R$ in the coordinates
$(z,t) \in \C^2 \times \R$ be given by
\begin{equation}
\begin{aligned}
U & =
\{
(z,t) \in \C^2 \times \R :
\norm{z} < 3, -1 < t < 3
\}
= B_3(0) \times (-1,3) ,
\\
K & =
\{
(z,t) \in \C^2 \times \R :
\norm{z} \leq 2, 1 \leq t \leq 2
\}
\\
& \phantom{{}={}} \cup
\{
(z,t) \in \C^2 \times \R :
1 \leq \norm{z} \leq 2, 0 \leq t < 1
\} .
\end{aligned}
\end{equation}
Then
\begin{equation}
\begin{aligned}
U \setminus K & =
\{
(z,t) \in \C^2 \times \R :
2 < \norm{z} < 3, 0 \leq t \leq 2
\}
\\
& \phantom{{}={}} \cup
\{
(z,t) \in \C^2 \times \R :
\norm{z} < 3, t \in (-1,0) \cup (2,3)
\}
\\
& \phantom{{}={}} \cup
\{
(z,t) \in \C^2 \times \R :
\norm{z} < 1, 0 \leq t < 1
\}
\\
& = U_1 \cup U_2 \cup U_3 .
\end{aligned}
\end{equation}
Define the smooth function $f \colon U \setminus K \to \C$ by
\begin{equation}
f(z,t)
=
\begin{cases}
0 & \text{if $(z,t) \in U_1$ or $U_2$} , \\
e^{-1/t^2} & \text{if $(z,t) \in U_3$} .
\end{cases}
\end{equation}
As $f$ is constant for fixed $t$, it is CR.
Any CR function on $U$ would have to be holomorphic on the entire leaf
$\{ (z,t) : \norm{z} < 3, t=\frac{1}{2} \}$.
Since the function $f$ is zero for
$\{ (z,t) : 2 < \norm{z} < 3, t=\frac{1}{2} \}$, any extension must be
identically zero on
$\{ (z,t) : \norm{z} < 3, t=\frac{1}{2} \}$,
while $f$ is nonzero on
$\{ (z,t) : \norm{z} < 1, t=\frac{1}{2} \}$.  Therefore
the extension cannot hold in the smooth case.

\begin{figure}[ht]
\includegraphics[width=2in]{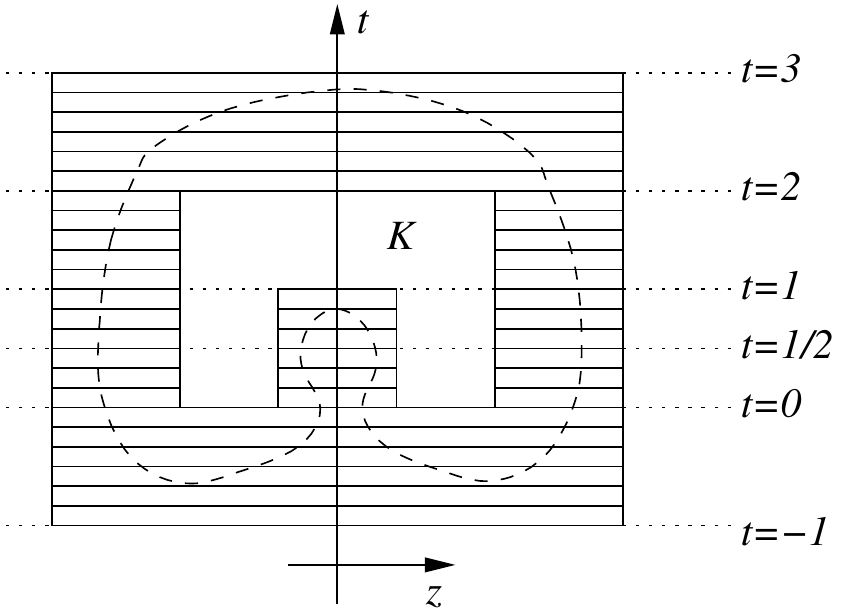}
\caption{Counterexample for smooth Hartogs with
parameters.\label{ctrex:fig}}
\end{figure}

Next we could fit a domain $\Omega$ so that $\partial \Omega$ lies inside
$U \setminus K$ in the way demonstrated in
Figure~\ref{ctrex:fig} by the dashed line.  As nondegenerate CR singularities are the
generic situation, we can ensure that $\partial \Omega$ only has such
singularities by taking a small perturbation if necessary.

Finally let us provide a sketch of a proof of Hartogs-Severi for $n \geq 2$.
We do this for two reasons:
In this case the argument is simpler than that found in the literature,
if we assume the solution to the standard Hartogs theorem in $\C^n$, and it
makes the present paper more self-contained.  We prove the following
statement:

\emph{Let $U \subset \C^{n} \times \R$, $n \geq 2$, be a bounded domain,
$K \subset \subset U$ a compact set such that $U \setminus K$
is connected, and $f \colon U \setminus K \to \C$ a real-analytic CR
function.  Then $f$ extends uniquely to
a real-analytic CR function of $U$.}

\begin{proof}[Sketch of proof]
Denote the variables by $(z,t) \in \C^n \times \R$.
Let $S_{t} = \{ z \in \C^n : (z,t) \in U \setminus K \}$, and
denote its topological components by $V_1,\ldots V_k$.
Let $S'_{t}$ be a union of the $V_j$ such that $V_j$ is not contained
in a compact component of $\C^{n} \setminus V_\ell$ for any other $\ell$.
Let $S''_{t}$ be the union of the components of $V_j$ that are in
$S'_{t}$ together with the compact components of $\C^{n} \setminus V_j$.

Then let $U' = \{ (z,t) \in \C^n : z \in S'_t \}$ and
$U'' = \{ (z,t) \in \C^n : z \in S''_t \}$.
It can be checked that $U'$ and $U''$ are open and $U' \subset U \setminus K
\subset U''$.
For each fixed $t$, using the standard Hartogs theorem on each component of $S'_t$
we extend $f$ restricted to $S'_t$ to a holomorphic function on $S''_t$.
That is we obtain a function $F \colon U'' \to \C$ that is holomorphic in
the $z$ variable and such that $f|_{U'} = F|_{U'}$.  We need to show
regularity and that $F$ equals to $f$ on all of $U \setminus K$.

Pick some $t_0$, and pick a compact
smooth hypersurface
$\Gamma$ in $S'_{t_0}$ around a compact component of the complement of
$S'_{t_0}$.
Inside $\Gamma$ we write $F(z,t_0)$ as the Bochner-Martinelli integral over
$\Gamma$.  Given that $\Gamma$ is compact we note that the
same $\Gamma$ will work for a small neighborhood of $t$ around $t_0$.
The Bochner-Martinelli kernel is real-analytic and hence $F(z,t)$ is real
analytic.  Therefore we have a real-analytic function $F|_{U}$ that agrees
with $f$ on $U'$.  By the uniqueness theorem
for real-analytic functions
and the connectedness of $U \setminus K$, we find that
$F|_{U}$ agrees with $f$ on $U \setminus K$.
\end{proof}


\def\MR#1{\relax\ifhmode\unskip\spacefactor3000 \space\fi%
  \href{http://www.ams.org/mathscinet-getitem?mr=#1}{MR#1}}


\end{document}